\documentclass[11pt]{amsart}

\usepackage[all,cmtip]{xy}
\usepackage{amsmath, amssymb}
\usepackage{mathtools}
\usepackage[hyphens]{url}
\usepackage[hidelinks]{hyperref}
\usepackage[T1]{fontenc}
\usepackage[english]{babel}

\usepackage[marginpar=2.5cm]{geometry}
\usepackage{xcolor}\usepackage{marginnote}
\usepackage{amsrefs}
\usepackage{graphicx}
\usepackage{subfig}
\usepackage{caption}
\usepackage{array}
\usepackage{float}

\newtheorem{thm}{Theorem}[section]

\newtheorem*{thm*}{Theorem}
\newtheorem{lem}[thm]{Lemma}

\newtheorem{prop}[thm]{Proposition}
\newtheorem*{prop*}{Proposition}
\newtheorem{conj}[thm]{Conjecture}

\theoremstyle{definition}
\newtheorem{defn}[thm]{Definition}
\newtheorem{notation}[thm]{Notation}
\newtheorem{remark}[thm]{Remark}

\newtheorem{problem}[thm]{Problem}

\def\e{\epsilon}
\def\la{\lambda}

\def\bb{\mathbb}

\def\om{\omega}

\def\de{\delta}

\def\bb{\mathbb}

\def\g{\gamma}

\def\cc{\mathcal}

\DeclareMathOperator{\tr}{tr}
\DeclareMathOperator{\Tr}{tr}

\DeclareMathOperator{\mal}{mal}
\DeclareMathOperator{\re}{Re}
\DeclareMathOperator{\im}{Im}

\DeclareMathOperator{\Var}{Var}

\newcommand\ip[2]{\left\langle #1\, , #2 \right\rangle}

\begin{document}

%%%%%%%%%%%%%%%%%%%%%%%%%%%%%%%%%%%%%%%%%%%%%%

\title{Malnormal Matrices}

\author[Mulcahy]{Garrett Mulcahy}
%\thanks{G. Mulcahy was supported by NSF grant DMS-1246818. }
\address{Department of Mathematics, University of Washington, Box 354350, Seattle, WA 98195-4350 }
\email{gmulcahy@uw.edu}

\author[Sinclair]{Thomas Sinclair}
%\thanks{T. Sinclair was supported by NSF grants DMS-1600857 and DMS-2055155.}
\address{Mathematics Department, Purdue University, 150 N. University Street, West Lafayette, IN 47907-2067}
\email{tsincla@purdue.edu}
\urladdr{http://www.math.purdue.edu/~tsincla/}

%%%%%%%%%%%%%%%%%%%%%%%%%%%%%%%%%%%%%%%%%%%%%%%%%%%%%%%%%%%%%%%%%%

\begin{abstract}
We exhibit an operator norm bounded, infinite sequence $\{A_n\}$ of $3n \times 3n$ complex matrices for which the commutator map $X\mapsto XA_n - A_nX$ is uniformly bounded below as an operator over the space of trace-zero self-adjoint matrices equipped with Hilbert--Schmidt norm. The construction is based on families of quantum expanders. We give several potential applications of these matrices to the study of quantum expanders. We formulate several natural conjectures and provide numerical evidence.
\end{abstract}

{\bf This article first appeared in the \emph{Proceedings of the American Mathematical Society} 150 (2022), no. 7, published by the American Mathematical Society.}\\

\setcounter{tocdepth}{1}
\maketitle

%\tableofcontents

\section{Introduction} In \cite{VonNeumann1942}, von Neumann demonstrated the following surprising result on the existence of matrix contractions in arbitrarily large dimension which uniformly poorly commute with all self-adjoint contractions of trace zero. We recall a representative result from that paper:

\begin{thm*}[von Neumann, {\cite[Theorem 9.7]{VonNeumann1942}}] For every $\delta>0$ there is an $\e>0$ so that for any $n\in \bb N$ there is a contraction $A\in \bb M_n$ which satisfies  $
    \|[A,B]\|_2<n^{1/2}\e \Rightarrow \|B\|_2<n^{1/2}\delta$
where $B\in\bb M_n$ is any self-adjoint contraction of trace zero.
\end{thm*}
Here, as throughout, $\bb M_n$ denotes the complex $n\times n$ matrices, $[A,B]$ denotes the commutator $AB-BA$ and $\|B\|_2$ denotes the Hilbert--Schmidt norm. This remarkable result has found several important applications such as to the theory of free probability \cite{Jung2007, Hayes2018} and to the model theory of II$_1$ factors \cite{FHS2014}. Moreover, von Neumann's interest in the ``finite, but very great'' \cite{VonNeumann1942} can be seen to anticipate the vibrant and rapidly developing field of asymptotic geometric analysis \cite{Milman2000, Vershik2007}.

Von Neumann's techniques are essentially probabilistic and are not constructive. He remarks \cite[Paragraph 11]{VonNeumann1942} that this is ``somewhat unsatisfactory,'' and that ``although our volumetric method seems to be quite powerful in securing existential results, it ought to be complemented by a more direct algebraic method, which names the resulting elements $A$ of $\bb M_n$ explicitly.'' The call for a constructive proof of von Neumann's result was more recently taken up by Vershik \cite[Remark 2]{Vershik2007}.

The goal of this note is to make progress in this direction. Our main result is the following variation on von Neumann's theorem:

\begin{thm}\label{main}
There is a universal constant $\g>0$ so that for infinitely many values $n\in\bb N$ there is a contraction $A\in \bb M_n$ satisfying
     $\|B\|_2\leq \g\|[A,B]\|_2$
where $B\in \bb M_n$ is any self-adjoint matrix of trace zero.
\end{thm}

We call a sequence of matrices satisfying the conclusion of Theorem \ref{main} \emph{malnormal}, and refer the reader to section \ref{sec:malnormal} below for a fuller treatment of this concept. The advantages to our approach are chiefly twofold. First, the dependency of $\e$ on $\delta$ in von Neumann's result, though explicit, is difficult to work out and does not appear to be linear, or even low-degree polynomial, while the dependency in our result is explicitly linear. Second, the proof of our result constructs a malnormal matrix given an input pair $U,V\in \bb U(n)$ of unitaries which form a so-called ``quantum expander.'' The existence of such pairs with uniform expansion constant for all $n$ suitably large is guaranteed by a result of Hastings \cite{Hastings2007} (see Lemma \ref{lem:hastings}), though this result is again non-constructive. This will be the starting point of our result, so the malnormal matrices are ultimately produced non-constructively. However, work of Ben-Aroya, Schwartz, and Ta-Shma \cite[Theorems 4.3 and 4.4]{BenAroya} and Harrow \cite{Harrow2008} shows how to construct explicit infinite families of uniform quantum expanders in $\bb U(n_k)^{D}$ with $n_k\to \infty$ given a suitable (and explicit) ``seed'' expander $U\in \bb U(n_0)^{d}$ with $D = d^2\geq 4$. We believe that our construction can be adapted to produce a malnormal matrix from any $d$-tuple of unitaries which form a quantum expander for $d\geq 2$ arbitrary, but we leave this as an open problem. Our result has, in fact, several connections with the theory of quantum expanders as developed in \cite{BenAroya2007, BenAroya, Hastings2007, Pisier2014}. As will be shown in section \ref{sec:q-expander} below, the malnormality of the matrix $(\re U + i\im V)/2$ implies that the pair $U,V\in \bb U(n)$ forms a ``quantum edge expander,'' a weaker notion than the aforementioned quantum expander. We do not know if the converse is true; see Remark \ref{mal-converse}. 

In light of this we investigate the following conjecture for the case of Haar random orthogonal matrices via numerical methods in section \ref{sec:experimental} and provide some positive evidence.

\begin{conj}\label{main-conj}
Let $U,V\in \bb U(n)$ be independently chosen Haar random unitaries and consider the random contraction $J := \frac{1}{2}\bigl(\re U  + i \im V\bigr).$ There is a universal constant $\gamma'>0$ so that with probability approaching $1$ as $n\to\infty$ the matrix $J$ satisfies
    $\|B\|_2\leq \gamma'\|[J,B]\|_2$
where $B\in \bb M_n$ is any self-adjoint matrix of trace zero.
\end{conj} 

The computations also suggest a positive answer to \cite[Conjecture 1.1]{Vershik2007} though the method of choosing random contractions differs from ours. There are, in fact, various ways one could go about sampling random contractions. The most straightforward would be to use the uniform (Lebesgue) measure on the set of contractions in $\bb M_n$ realized in $\bb R^{2n^2}$, but this is difficult to work with in practice. Indeed, von Neumann would have liked to work directly with such random contractions \cite[Paragraph 11]{VonNeumann1942}. He instead resorts to considering the more tractable uniform measure on the $\|\cdot\|_2$-norm unit ball, deploying his formidable analytic skill to deal with the fact that such matrices can have arbitrarily large singular values as $n\to \infty$. The matrix $J$ as above is a guaranteed contraction as it is randomly chosen under the push-forward of the Haar probability measure on $\bb U(n)^2$ under the map $(U,V)\mapsto (U+U^*+V-V^*)/4$. It is easy to simulate Haar random unitaries, as this requires only a means of choosing an $n\times n$ array $Y$ of independent complex gaussian random variables. The unitary in the polar decomposition of $Y = U|Y|$ is Haar distributed \cite[Appendix A]{Pisier2014}. In the last section, devoted to open problems, we consider using the Ginibre ensemble, which does not produce a random contraction, but does produce a matrix with largest singular value at most $2$ with probability tending to $1$ as $n\to\infty$; this follows from \cite[Theorem 2.11]{DavidsonSzarek2003}, for instance.

%%%%%%%%%%%%%%%%%%%%%%%%%%%%%%%%%%%%%%%%%%%%%%%%%%%%%%%%%%%%%%%%%%

\section{Preliminaries}\label{sec:prelim}

    For $A = (A_{ij})\in \bb M_n$ we denote the usual matrix trace by $\tr$ and the normalized trace of $A$ as $\tau(A)$, given by $\tau(A) := \frac{1}{n}\sum_i A_{ii}.$
    Note that $\tau(I_n) =1$, where $I_n$ is the identity matrix.
The \emph{Hilbert--Schmidt norm} of a matrix $A\in \bb M_n$ is defined as
    \begin{equation*}
        \|A\|_2 := \tr(A^*A)^{1/2} = \left(\sum_{i,j} |A_{ij}|^2\right)^{1/2}.
    \end{equation*}
This is the norm on $\bb M_n$ corresponding to the inner product $\ip{A}{B} := \tr(B^*A) = \sum_{i,j} A_{ij}\overline{B_{ij}}$. Note that $\|A\|_2 = \|A^*\|_2$. It is easy to see from the definition and unitary invariance of the trace that for $A\in \bb M_n$ and $U,V\in \bb U(n)$ we have that
\begin{equation}\label{unitary-invariance}
    \|UAV\|_2 = \|A\|_2.
\end{equation} 
It is a standard fact for $A,B\in \bb M_n$ that $\|AB\|_2\leq \|A\|\|B\|_2$ where $\|A\|$ is the operator norm, i.e., the largest singular value of $A$ \cite[Proposition IV.2.4]{Bhatia1997}.

%\subsection{Quantum Expanders}

Let $U = (U_1,\dotsc,U_k)\in \bb U(n)^k$ be a $k$-tuple of unitaries and define 
\begin{equation}
    \cc E_U(X) := \frac{1}{k}\sum_{i=1}^k U_i^*XU_i\quad,\quad \cc E_U^\dagger(X) := \frac{1}{k}\sum_{i=1}^k U_iXU_i^*\quad,\quad \cc E_U^h := \frac{1}{2}(\cc E_U + \cc E_U^\dagger)
\end{equation}
which are trace-preserving, unital, completely positive maps. It is easy to check that
    $\tr(\cc E_U(X)Y) = \tr(\cc E_U^\dagger(Y)X)$
for all $X,Y\in\bb M_n$; hence,
\begin{equation}\label{Eh-equality}
    \tr(\cc E_U(B) B) = \tr(\cc E_U^h(B)B) = \tr(\cc E_U^\dagger(B)B)
\end{equation} for all $B$ self-adjoint. 

\begin{notation} For a matrix $B$, let $\dot B := B - \tau(B)I$. In particular $\tr(\dot B) =0$ and $\|[X,B]\|_2 = \|[X,\dot B]\|_2$. Let $\cc H_n$ (respectively, $\cc H_n^0$) be the subspace of $n\times n$ self-adjoint matrices (resp., self-adjoint matrices of trace zero) equipped with the Hilbert-Schmidt norm.
\end{notation}

We recall the definitions of a quantum expander and a quantum edge expander due to Hastings \cite{Hastings2007}.

\begin{defn}\label{defn:q-expander}
 We will say that $U =(U_1,\dotsc,U_k)$, a $k$-tuple of unitaries, is a \emph{quantum $\de$-edge expander} for some $0\leq\delta <1$ if we have that $\cc E_U$ satisfies
\begin{equation}\label{edge-expander-1}
    \tr(\cc E_U(\dot B)\dot B)\leq \delta\tr(\dot B^2)
\end{equation}
for all $B$ self-adjoint. Likewise, we say that $U$ is a \emph{quantum $\de$-expander} if 
\begin{equation}\label{q-expander-1}
    \|\cc E_U: \cc H_n^0\to \cc H_n^0\|\leq \de
\end{equation}
\end{defn} 

Equivalently, we have that $U$ is a quantum $\delta$-edge expander if and only if
\begin{equation}\label{edge-expander-1-1}
    \tau(\cc E_U(B)B)\leq \delta\, \tau(B^2) + (1-\delta)\,\tau(B)^2
\end{equation}
for all $B$ self-adjoint. Using the identity
\begin{equation}\label{commute-identity}
    \frac{1}{k}\sum_{i=1}^k\|[U_i,B]\|_2^2 = 2\|\dot B\|_2^2 -2\tr(\cc E_U(\dot B)\dot B)
\end{equation}
the conditions (\ref{edge-expander-1}) and (\ref{edge-expander-1-1}) are again equivalent to 
\begin{equation}
     \frac{1}{k}\sum_{i=1}^k\|[U_i,B]\|_2^2\geq 2(1-\delta)\|\dot B\|_2^2
\end{equation}
for all $B$ self-adjoint.

    In Hastings' original formulation \cite[Appendix]{Hastings2007} a quantum $\la$-edge expander $U$ is a tuple of unitaries from $\bb U(n)$ so that 
    \begin{equation}\label{edge-expander-3}
        \tr(\cc E_U(P)P)\leq \la\, \tr(P)
    \end{equation}
    for all projections $P\in \bb M_n$ of rank at most $n/2$. Using (\ref{edge-expander-1-1}) it follows that a quantum $\de$-edge expander satisfies (\ref{edge-expander-3}) with $\la\leq (1+\de)/2$. Conversely, by \cite[(A11)]{Hastings2007} it follows that the constant $\de$ in (\ref{edge-expander-1}) can be chosen to satisfy $1-\la \leq \sqrt{2(1 - \de)}$.

\begin{remark}
    Note that $\|\cc E_U^h: \cc H_n^0\to \cc H_n^0\|\leq \|\cc E_U: \cc H_n^0\to \cc H_n^0\|$. Since $\cc E_U^h:\cc H_n^0\to \cc H_n^0$ is self-adjoint, we have that \[\|\cc E_U^h: \cc H_n^0\to \cc H_n^0\| = \sup_{\|B\|_2=1, B\in \cc H_n^0} |\tr(\cc E_U^h(B)B)|\geq^{eq. (\ref{Eh-equality})} \sup_{\|B\|_2=1, B\in \cc H_n^0} \tr(\cc E_U(B)B).\] Thus $U$ being a quantum $\de$-edge expander is weaker than $U$ being a quantum $\de$-expander.
\end{remark}

The following result is due to Hastings \cite{Hastings2007} (cf.~\cite[Lemma 1.8 and Theorem 4.2]{Pisier2014}).

\begin{lem}\label{lem:hastings}
Let $U = (U_1,\dotsc, U_k)\in \bb U(n)^k$ be a $k$-tuple of unitaries sampled according to Haar measure for some $k\geq 2$. For all $\e>0$ we have that 
\[\|\cc E_U^h : \cc H_n^0\to \cc H_n^0\|\leq \frac{\sqrt{2k-1}}{k} + \e\]
with probability tending to $1$ as $n\to \infty$.
\end{lem}

%We will make use of the following 
%\begin{cor}
%For all $n$ sufficiently large, there exist $U,V,W\in \bb U(n)$ which form a quantum $\delta$-expander for $\de= 0.95$.
%\end{cor}

%%%%%%%%%%%%%%%%%%%%%%%%%%%%%%%%%%%%%%%%%%%%%%%%%%%%%%%%%%%%%%%%%%%%%%%%%%%%%%%%%%%%%
%%%%%%%%%%%%%%%%%%%%%%%% MALNORMAL %%%%%%%%%%%%%%%%%%%%%%%%%%%%%%%%%%%%%%%%%%%%%%%%%
%%%%%%%%%%%%%%%%%%%%%%%%%%%%%%%%%%%%%%%%%%%%%%%%%%%%%%%%%%%%%%%%%%%%%%%%%%%%%%%%%%%%%

\section{Malnormal Matrices}\label{sec:malnormal}

%We begin with some basic remarks on malnormal matrices.

\begin{defn}
A matrix $X\in \bb M_n$ is \emph{$\kappa$-malnormal} if 
\begin{equation}
     \|[X,B]\|_2\geq \kappa \|B - \tau(B)I\|_2
\end{equation}
for all self-adjoint matrices $B\in \cc H_n$. We define $\mal(X)$ to be the largest constant $\kappa$ for which $X$ is $\kappa$-malnormal, and we simply say that $X$ is \emph{malnormal} if $\mal(X)>0$.
\end{defn}

Since $\|[X,B]\|_2 = \|[X,\dot B]\|_2$, $X$ being $\kappa$-malnormal is equivalent to the linear operator $B\mapsto [X,B]$ being $\kappa$-bounded from below on $\cc H_n^0$ equipped with the Hilbert-Schmidt norm. Moreover, by compactness of the unit ball in $\cc H_n^0$, $\mal(X)$ is attained; hence, $\mal(X)=0$ if and only if $X$ commutes with a non-scalar self-adoint matrix if and only if $X$ and $X^*$ commute with a common non-scalar matrix. The term `malnormal' is borrowed from group theory where a subgroup $H$ of a group $G$ is said to be malnormal if $gHg^{-1} \cap H = \{1\}$ for all $g\in G\setminus H$ \cite[p. 203]{Lyndon2001}. Similary, $X$ being malnormal is a strong negation of $X$ being normal as any normal matrix commutes with its real and imaginary parts.

As we will see below, malnormality is really intended as a quantitative concept, that is, the \emph{size} of the constant $\mal(X)$ is more pertinent than its mere existence. More precisely, there seem to be many explicit constructions of families of contractions $X_n\in \bb M_n$ for which $\mal(X_n)^{-1} = O(\sqrt{n})$ (for instance, the shift matrices $(S_n)_{i,j} = \delta_{i,j-1}$), while it seems to be a nontrivial task to produce examples which even satisfy $\mal(X_n)^{-1} = o(\sqrt n)$. Thus this property seemingly captures some genuinely new phenomenon about the matrix when $\mal(X)^{-1}\ll \sqrt n\|X\|$.

\subsection{Proof of the Main Result}
\begin{prop}
There exists a constant $c>0$ so that for each $n\in\bb N$ sufficiently large, there is a contraction $X_n\in \bb M_{3n}$ satisfying $\mal(X_n)\geq c$.
\end{prop}

\begin{proof}
Let us fix $n\in\bb N$ sufficiently large and choose $U,V \in \bb U(n)$ to be a quantum $\delta$-expander. Such a pair of unitaries is guaranteed to exist for $n$ sufficiently large and $\delta=0.87$ by Lemma \ref{lem:hastings}. Consider the matrix
\begin{equation*}
    \cc X = \begin{pmatrix}  0 & 2 U & 0\\  0 & 0 & V\\  3 I_n & 2 I_n & I_n\end{pmatrix}\in \bb M_{3n}.
\end{equation*}
Clearly, $\|\cc X\|\leq 9$. Let $\cc B\in \cc H_{3n}^0$ be a unit vector which we will write as
\begin{equation*}
    \cc B = \begin{pmatrix} A & X & Y \\ X^* & B & Z \\ Y^* & Z^* & C 
    \end{pmatrix}
\end{equation*}
with $A,B,C,X,Y,Z\in \bb M_n$.

Now let us compute
\begin{equation}\label{2Uright-compose}
    \cc X\cc B = \begin{pmatrix} 2UX^* & 2UB & 2UZ \\ VY^* & VZ^* & VC \\ 3A+2X^*+Y^* & 3X+2B+Z^* & 3Y+2Z+C 
    \end{pmatrix}
\end{equation}
and 
\begin{equation}\label{2Uleft-compose}
    \cc B\cc X = \begin{pmatrix} 3Y & 2AU + 2Y & XV+Y \\ 3Z & 2X^*U + 2Z & BV + Z \\
    3C & 2Y^*U + 2C & Z^*V + C 
    \end{pmatrix}.
\end{equation}

Before proceeding we will pause to introduce some convenient

\begin{notation}
For $a,b\in \bb R$ we will write $a\leq_\e b$ if $a\leq b+\e$ and $a=_\e b$ if $a\leq_\e b$ and $b\leq_\e a$, i.e, $|a-b|\leq \e$. For $X,Y\in \bb M_n$ we will write $X =_\e Y$ to denote $\|X-Y\|_2\leq \e$.
\end{notation}
In what  follows we will denote $a := \|A\|_2, x := \|X\|_2$, and analogously for the norms of all block components of $\cc X$. The notation of $[(i,j), (k,l)]$ indicates that the approximate identities are obtained by comparing the $(i,j)$-blocks and the $(k,l)$-blocks, in that order, of the matrices (\ref{2Uright-compose}) and (\ref{2Uleft-compose}).

Let us fix $\e>0$ sufficiently small and suppose that $\|[\cc X,\cc B]\|_2\leq \e$.
From the $(2,1)$ and $(2,2)$-blocks and the (1,1)-block we have that
% \begin{flalign}
%     && && &Y =_{\e} 3Z^*V =_{3\e} 6U^*X+2Z^*, && [(2,1), (2,2)]\\
%   && && &3Y =_{\e} 2UX^*; && [(1,1)]
% \end{flalign}
% [I reworked the equations and got this instead:
\begin{flalign}
    && && Y &=_{\e} 3Z^*V =_{3\e} 6V^*X^*UV+6V^*Z^*V && [(2,1), (2,2)]\\
   && && & =_{3\e} 9V^*U^*YUV + 6V^*ZV; && [(1,1)]
\end{flalign}
% This should show that $8y \leq_{7\e} 6z$, which when combined with (2,1) ought to give $8y \leq_{9\e} 2y$ or $y\leq 2\e$. The estimates for $x$ and $z$ should follow from this.]
%\[4U^*R^*UV + 3QV =_\e 3U^*XV + 3QV =_\e 2Y =_{2\e} 2QW^* - 4VZW^* =_\e 2QW^* - VWR^*W^*\]
hence, by several applications of the triangle inequality and (\ref{unitary-invariance}), $8y \leq 6z + 7\e$ as follows,
\begin{equation}\label{yz}
         8y = \|9V^*U^*YUV\|_2 - \|Y\|_2 \leq \|9V^*U^*YUV - Y\|_{2} \leq_{7\e} \|6V^*ZV\|_2
         = 6z.
\end{equation}
The (2,1)-block gives us that
\begin{equation}\label{zy}
    3z - y \leq \|Y-3Z^{*}V\|_{2} \leq \e,
\end{equation}
so by combining (\ref{yz}) and (\ref{zy}) we have that
\begin{equation}\label{y}
    8y \leq 6z+7\e \leq 2y + 9\e  \ \Rightarrow\ y \leq \frac{3}{2}\e.
\end{equation}
From this bound on $y$, we can use the (2,1)-block to deduce
\begin{equation}\label{z}
    3z \leq y + \e \leq \frac{3}{2}\e + \e \ \Rightarrow z \leq \frac{5}{6}\e.
\end{equation}
Lastly, from the (1,1)-block we get
\begin{equation}\label{x}
    2x \leq 3y + \e \leq \frac{9}{2}\e + \e \ \Rightarrow x \leq \frac{11}{4}\e.
\end{equation}

Hence, we have deduced that
\begin{flalign}\label{2U200}
   && x,y,z\leq 3\e. && 
\end{flalign}
Setting 
\begin{equation*}
    \cc B_0 = \begin{pmatrix} A & 0 & 0 \\ 0 & B & 0 \\ 0 & 0 & C 
    \end{pmatrix},
\end{equation*}
using (\ref{2U200}) the triangle inequality implies that
\begin{equation}\label{2Ucontrol-b0}
    \|\cc B_0\|_2 \geq \|\cc B\|_2 - \|\cc B - \cc B_0\|_2\geq 1 - 6(3\e) \geq 1 - 10^2\e.
\end{equation}

Since 
\begin{equation*}
    [\cc X, \cc B_0] = \begin{pmatrix} 0 & 2UB & 0 \\ 0 & 0 & VC \\ 3A & 2B & C \end{pmatrix} - \begin{pmatrix} 0 & 2AU & 0 \\ 0 & 0 & BV \\ 3C & 2C & C\end{pmatrix}
\end{equation*}
it follows by (\ref{2Uright-compose}) and (\ref{2Uleft-compose}) combined with (\ref{2U200}) that
\begin{equation}\label{2Ucontrol-c}
    \|[\cc X,\cc B_0]\|_2\leq \|[\cc X, \cc B]\|_2 + \|[\cc X, \cc B] - [\cc X,\cc B_0]\|_2 \leq \e + 36(3\e) \leq 10^3\e.
\end{equation}
Examining the entries of $[\cc X,\cc B_0]$ it follows that
\begin{equation}\label{abc-almost-equal}
\begin{cases}
    \|A - C\|_2, \|B-C\|_2 \leq 10^3\e,\\
    \|B - U^*AU\|_2, \|C - V^*BV\|_2 \leq 10^3\e;
\end{cases}
\end{equation}
hence, by the first line and the triangle inequality 
\begin{equation*}
    \|A - B\|_2, \|A - C\|_2 \leq (2\times 10^3)\e.
\end{equation*}
Along with the second line of (\ref{abc-almost-equal}) and (\ref{unitary-invariance}) this implies that
\begin{flalign}\label{2Ualmost-commute}
   && && \|A - S^*AS\|_2\leq (5\times 10^3)\e, && S=U,V,U^*,V^*.
\end{flalign}
Here is one case; the others are similar:
\begin{equation*}
    \begin{aligned}
    \|A - V^*AV\|_2 &\leq \|A - C\|_2 + \|C - V^*BV\|_2 + \|V^*BV - V^*AV\|_2\\
    &=\|A - C\|_2 + \|C - V^*BV\|_2 + \|B - A\|_2 \leq (5\times 10^3)\e.
    \end{aligned}
\end{equation*}
Setting $\cc E(A) = \frac{1}{2}(U^*AU + V^*AV)$ and using (\ref{2Ualmost-commute}) this implies that
\begin{equation}
    \begin{aligned}
        \|A - \cc E^h(A)\|_2\leq 10^4\e.
    \end{aligned}
\end{equation}

By assumption $\|\cc E^h: \cc H_n^0\to \cc H_n^0\|\leq \de<1$; hence, by the triangle inequality 
\begin{equation}
    \begin{aligned}
        (1-\delta)\|\dot A\|_2 &\leq \|\dot A\|_2 - \|\cc E^h(\dot A)\|_2 \leq \|\dot A - \cc E^h(\dot A)\|_2\\
        &= \|(A - \tau(A)I_n) - \cc E^h(A - \tau(A)I_n)\|_2\\
        &= \|A - \tau(A)I_n + \tau(A)I_n - \cc E^h(A)\|_2 = \|A - \cc E^h(A)\|_2
    \end{aligned}
\end{equation} since $\cc E^h(I_n) = I_n$. Thus, we have that
\begin{equation}\label{2Ua-almost-trace}
    \|A - \tau(A)I_n\|_2 = \|\dot A\|_2 \leq 10^4\e/(1 -\delta).
\end{equation}
Recall that for $X\in \bb M_n$ the map $X\mapsto \tau(X)I_n$ is the orthogonal projection of $X$ onto $\bb CI_n$ from which it follows that $\|X - \tau(X)I_n\|_2 = \inf_{c} \|X - c\, I_n\|_2$. Combining (\ref{2Ucontrol-c}), (\ref{abc-almost-equal}), and (\ref{2Ua-almost-trace})  this shows that 
\begin{equation}
    1 -  10^2\e \leq \|\cc B_0\|_2 = \|\cc B_0 - \tau(\cc B_0)I_{3n}\|_2\leq \|\cc B_0 - \tau(A)I_{3n}\|_2\leq 10^5\e/(1-\delta)
\end{equation}
which is impossible for $\e>0$ chosen suitably small. Thus there is a uniform constant $c>0$ depending only on $\delta$ so that $\mal(\cc X)> c = c(\delta)$. \qedhere
\end{proof}

%%%%%%%%%%%%%%%%%%%%%%%%%%%%%%%%%%%%%%%%%%%%%%%%%%%%%%%%%%%%%%%%%%

\subsection{Connections with Quantum Edge Expanders}\label{sec:q-expander}

As seen in the proof of the main result, malnormal matrices have close connections to the theory of quantum expanders. We take the opportunity to sketch out several more ways in which the theories are connected.

\begin{lem}\label{lem:split-j}
Let $X\in \bb M_n$ and $B\in \cc H_n$, then \begin{equation}
    \|[X,B]\|_{2}^2 = \|[\re X,B]\|_{2}^2 + \|[\im X, B]\|_{2}^2.
\end{equation}
\end{lem}

\begin{proof} Since $B$ is self-adjoint we have that
\[ \|[X,B]\|_{2}^2 = \Tr([X,B]^*[X,B]) = -\Tr([X^*,B][X,B]).\]
Writing $X = S + i T$, with $S,T$ self-adjoint, we have 
\[ \Tr([X^*,B][X,B]) = \Tr([S,B]^2) + \Tr([T,B]^2) - i\Tr([T,B][S,B]) + i\Tr([S,B][T,B])\] from which it follows that $\|[X,B]\|_{2}^2 = \|[S,B]\|_{2}^2 + \|[T,B]\|_{2}^2$. \qedhere
\end{proof}

\begin{prop} Let $(U,V)$ be a pair of unitaries in $\bb M_n$ and set $J= \frac{1}{2}\bigl(\re U + i \im V\bigr)$. If $J$ is malnormal, then $(U,V)$ form a quantum edge expander with constant $\delta \leq 1 - \mal(J)^2$.

\end{prop}

\begin{proof}
Since $J$ is malnormal there is some $\kappa>0$ so that $\|[J,B]\|_2^2\geq \kappa^2\|\dot B\|_2^2$ for all $B$ self-adjoint. We fix $B$ self-adjoint. It follows from two applications of Lemma \ref{lem:split-j} that
\begin{equation}\label{j-commute}
    4\|[J,B]\|_2^2 = \|[\re U,B]\|_2^2 + \|[\im V,B]\|_2^2\leq \|[U,B]\|_2^2 + \|[V,B]\|_2^2.
\end{equation}
Combining (\ref{j-commute}) with (\ref{commute-identity}), we have
\begin{equation}
    \begin{aligned}
        \kappa^2\|\dot B\|_2^2\leq \|[J,B]\|_2^2\leq  \|\dot B\|_2^2 - \tr(\cc E(\dot B)\dot B),
    \end{aligned}
\end{equation}
where $\cc E(X) = \frac{1}{2}(U^*XU + V^*XV)$ as above. Thus $\tr(\cc E(\dot B)\dot B)\leq (1-\kappa^2)\tr(\dot B^2)$ and $\cc E$ is a quantum edge expander. \qedhere
\end{proof}

\begin{remark}\label{mal-converse}
    In terms of a converse, even assuming that $(U,V)$ is a quantum expander, we were not able to determine whether $J$ is malnormal with effective control on $\mal(J)$, so this remains an interesting open problem.
\end{remark}

Any contraction $X\in\bb M_n$ may be decomposed as an average of two unitaries, say $X = (U+V)/2$. We do not know whether $X$ being malnormal implies that such $(U,V)$ can be chosen to form a quantum (edge) expander. However, the following result lies in this direction.

\begin{prop}
Let $U_1, \dotsc, U_k$ be unitaries in $\bb M_n$. For $\om\in \bb T^k$ let $J_\om := \frac{1}{\sqrt {2k}}\sum_{i=1}^k \om_i U_i$. Let $\la^k$ be the Haar probability measure on $\bb T^k$. If $J_\om$ is malnormal $\la^k$-almost surely, then $(U_1,\dotsc, U_k)$ form a quantum edge expander.
\end{prop}

\begin{proof}
Let $\kappa(\om) = \mal(J_\om)$ and set $\kappa := \left(\int_{\bb T^k} \kappa(\om)^2\, d\la^k(\om)\right)^{1/2}$. Fixing $B$ self-adjoint we have that 
\begin{equation}
    \begin{aligned}
        \kappa^2\|\dot B\|_2^2 \leq \int_{\bb T^k} \|[J_\om,B]\|_2^2\, d\la^k(\om) &= \frac{1}{2k}\sum_{i,j=1}^n \int_{\bb T^k} \om_i\bar\om_j \tr([U_i,B][U_j,B]^*) d\la^k(\om)\\
        &= \frac{1}{2k}\sum_{i=1}^k \|[U_i,B]\|_2^2 = \|\dot B\|_2^2 - \tr(\cc E(\dot B)\dot B)
    \end{aligned}
\end{equation}
where $\cc E(X) = \frac{1}{k}\sum_{i=1}^k U_i^*XU_i$. The last equality is (\ref{commute-identity}). This proves the claim. \qedhere
\end{proof}

%%%%%%%%%%%%%%%%%%%%%%%%%%%%%%%%%%%%%%%%%%%%%%%%%%%%%%%%%%%%%%%%%%%%%%%%%%%%%%%%%%%%%%%%%%%%%%%%%%%%%%%%%%%
%%%%%%%%%%%%%%%%%%%%%%%%%%%%%%%%%%%%%%

\section{Experimental Methods and Results}\label{sec:experimental}
As motivated in Section \ref{sec:q-expander}, we wish to investigate the distribution of $\mal(J)$ for matrices of the form $J= \frac{1}{2}\bigl(\re(U) + i \im(V)\bigr)$, where $(U,V)$ are random unitary matrices sampled independently according to the Haar measure. We also wish to study the asymptotics of this distribution as $n$ tends to infinity.

{\bf N.B.} In order to reduce the number of variables to optimize over, we only consider the real-valued equivalent of this problem; that is, for the remainder of this section we investigate matrices of the form $J = \frac{1}{2}(\re(U)+i\im(V))$, where $(U,V)$ are independent Haar random \emph{orthogonal} matrices.

For a given matrix $X$, the quantity $\mal(X)$ is the solution to the following optimization problem.
\begin{equation} \label{op_prob}
\begin{aligned}
& \underset{B \in \mathbb{R}^{n \times n}}{\text{minimize}} 
& & \|XB-BX\|_{2} \\
& \text{subject to}
& &  \|B\|_{2} -1 = 0 \\
& & & \tr(B) = 0 \\
\end{aligned}
\end{equation}

We will now present ($\ref{op_prob}$) as an optimization problem over vectors as opposed to matrices. Let $d(n) := \dim(\mathcal{H}_{n}^{0}) = \frac{n(n+1)}{2}-1$, then $\mathcal{H}_{n}^{0}$ and $\mathbb{R}^{d(n)}$ are isometrically isomorphic as inner product spaces. Let $\varphi : \mathbb{R}^{d(n)} \to \mathcal{H}_{n}^{0}$ be an isometry, that is, an isomorphism that preserves the norm between $\mathbb{R}^{d(n)}$ and $\mathcal{H}_{n}^{0}$.
%, i.e.\ that $\|x\| = \|\varphi(x)\|_{2}$ for all $x \in \mathbb{R}^{d(n)}$. 
We now consider $f : \mathbb{R}^{d(n)} \to \mathbb{R} $ defined as
\begin{equation}
    f(b) := \|X \varphi(b) - \varphi(b) X \|_{2}^{2}.
\end{equation}

Since $f$ is a sum of squares, it is a quadratic function with the form
\begin{equation}\label{quad-op-prob}
    f(b) = \frac{1}{2} b^{T}H(X)b = \| X \varphi(b) - \varphi(b) X \|_{2}^{2}
\end{equation}
where $H(X) \in \mathbb{M}_{d(n)}$ is the Hessian of $f$ whose entries are completely determined by the fixed matrix $X$. Since $f$ is $\mathcal{C}^{\infty}$ and non-negative, $H(X)$ is symmetric and positive semidefinite. Lastly, since $\varphi$ is an isometry, the first equality constraint of ($\ref{op_prob}$) that $\|\varphi(b)\|_{2} - 1 =0$ is equivalent to $b^{T}b-1 = 0$ for all $b \in \mathbb{R}^{d(n)}$.  Let $\la_{1}$ denote the smallest eigenvalue of $H(X)$, then the global minimum of ($\ref{quad-op-prob}$) subject to $b^{T}b - 1 =0$ is $\la_1 /2$, and thus $\mal(X)= \sqrt{\la_{1}/2}.$ Further, any global minimizer is then an eigenvector of $H(X)$ corresponding to $\la_{1}$. Setting $H = H(X)$ we see that the problem stated in ($\ref{op_prob}$) is a special case of the more general ``quadratic over a sphere'' problem \cite{Hager2001}:
%Thus, we arrive at the following optimization problem identical to ($\ref{op_prob}$):
\begin{equation} \label{gen_op_prob}
\underset{x \in \mathbb{R}^{n}}{\text{minimize}}\ \,
\frac{1}{2} x^{T}Hx,\quad\text{s.t.}\ \ x^{T}x - 1 = 0 \\
\end{equation}
for $H$ a hermitian $n\times n$ (real) matrix.

We now establish the following result that will be essential in using optimization algorithms to find constants of malnormality. The result is probably well-known, but we include a proof here for the sake of completeness. 

\begin{prop} \label{opt_thm} 
A vector $x^{*} \in \mathbb{R}^{n}$ is a local minimizer of $(\ref{gen_op_prob})$ if and only if it is a global minimizer of $(\ref{gen_op_prob})$.
\end{prop}

\begin{proof}
If $x^{*}$ is a global minimizer then the conclusion is immediate, so suppose that $x^{*}$ is a local minimizer of ($\ref{gen_op_prob}$). By the Second Order Necessary conditions as stated in \cite[Theorem 9.3.1]{Fletcher2000} there exists $\xi^{*} \in \mathbb{R}$ such that
\begin{align*}
    Hx^{*} + 2 \xi^{*} x^{*} &= 0 \\
    (x^{*})^{T}x^{*} - 1 &= 0 \\
    d^{T} (H + 2 \xi^{*} I) d &\geq 0, \forall d\ \textup{s.t.}\ (x^{*})^{T}d = 0.
\end{align*}

The first two conditions imply that $-2\xi^*$ is an eigenvalue.
For a contradiction, suppose that there exists an eigenvalue $\la^*$ of $H$ such that $\lambda^{*} < -2\xi^{*}$. Let $v^{*}$ be a unit eigenvector corresponding to $\lambda^{*}$, then since $v^{*}$ and $x^{*}$ are in different eigenspaces of $H$, we must have that $(x^{*})^{T}v^{*} = 0$. However,
\begin{equation*}
    (v^{*})^{T}(H + 2 \xi^{*} I)v^{*} =\lambda^{*} + 2\xi^{*} < 0,
\end{equation*}
which contradicts the third stipulation of the second order necessary conditions. Hence, $-2\xi^{*}$ must be the smallest eigenvalue of $H$. Since $x^{*}$ is then an eigenvector corresponding to the smallest eigenvalue of $H$, we have that $x^{*}$ is the global minimizer of ($\ref{gen_op_prob}$).
\end{proof}

\subsection{Numerical Implementation}
We utilized two approaches to calculate a matrix's constant of malnormality. For small dimension, we computed the Hessian $H(X)$ and then computed its smallest eigenvalue. The other approach involved using MATLAB's optimization toolbox to solve the optimization problem outlined in ($\ref{gen_op_prob}$). Both approaches rely on constructing some explicit isometry $\varphi$ between $\mathbb{R}^{d(n)}$ and $\mathcal{H}_{n}^{0}$. 
The precise details of our implementation are described on the project's GitHub repository \cite{githubMalMat}.%\footnote{\href{https://github.com/garrett-mulcahy/Malnormal-Matrices/ }{https://github.com/garrett-mulcahy/Malnormal-Matrices/ }}
%, which we describe in [CITE github document] we which will now describe. For a fixed dimension $n$, we map the first $n-1$ elements in the standard ordered basis for $\mathbb{R}^{d(n)}$ to diagonal trace-zero matrices in $\mathcal{H}_{n}^{0}$ whose diagonal entries are determined by the following vectors, suitably normalized:

%We will now illustrate our first approach. 
For the first approach, given a general matrix $X \in \mathbb{M}_{n}$ and vector $b \in \mathbb{R}^{d(n)}$ we used the MATLAB symbolic toolbox \cite{matlabsymbolic} to compute $f(b) = \| X\varphi(b)-\varphi(b)X \|_{2}^{2}$. Then, we symbolically differentiated $f$ to obtain the Hessian $H(X)$ as a function of the entries of the fixed matrix $X$. Lastly, we evaluate the Hessian for a large number of matrices $X$, computing $\mal({X}) = \sqrt{\la_{1}/2}$ for each matrix. 

However, the Hessian files created in the aforementioned procedure became exceedingly large very quickly; thus, to compute a large enough quantity of malnormality constants we turned to MATLAB's optimization toolbox \cite{matlaboptimization}. For a given matrix $X$ we posed the problem (\ref{gen_op_prob}) with $H= H(X)$. 
As Proposition $\ref{opt_thm}$ shows, if the algorithm converges to a local minimum, then that minimum is the global minimum. Hence, the square root of that global minimum is $\mal({X})$. MATLAB returns an exit flag indicating whether or not the algorithm converged to a local minimum (with respect to a set of specified tolerances); thus, as long as an exit flag indicating convergence to  a local minimum was returned, we accepted the value as a global minimum.

\subsection{Results}
Let $\mathcal{J}_{n} = \left\{(U+U^{*}+V-V^{*})/4: U,V \in \bb O(n) \right\}$ be equipped with the push-forward of the product Haar measure on $\bb O(n)\times \bb O(n)$, and let $\mal(\mathcal{J}_{n})$ denote the distribution of the constants of malnormality of $\mathcal{J}_{n}$. To study these distributions, we compute the constant of malnormality for a large number of matrices generated from $\mathcal{J}_{n}$ for $n = 3, \dots 30$. For $n \leq 17$, we used the first approach outlined previously (direct computation of the Hessian). For $n \geq 18$ we used MATLAB's {\sf fmincon} solver with the interior-point algorithm and only included values which had an exit flag reporting convergence to a minimum.

%Aside from when $n=3,4, \textup{and } 5$, the distributions of the constant of malnormalities tend to resemble the shape of the above distribution, with similar skew. 
To see how the distribution of $\mal(\mathcal{J}_{n})$ changes as $n$ increases, we use plot the densities of the empirical distributions for $n=10,15,20,25,30$ on the same axis. We obtained these distributions by fitting a kernel smoothing function with Gaussian kernel, which allowed us to avoid introducing any distributional assumptions on the data. 
%We summarize our data generation in Table $\ref{summarytable_malJ}$ in the appendix.

\begin{figure}[ht]
\centering
\includegraphics[scale=0.5]{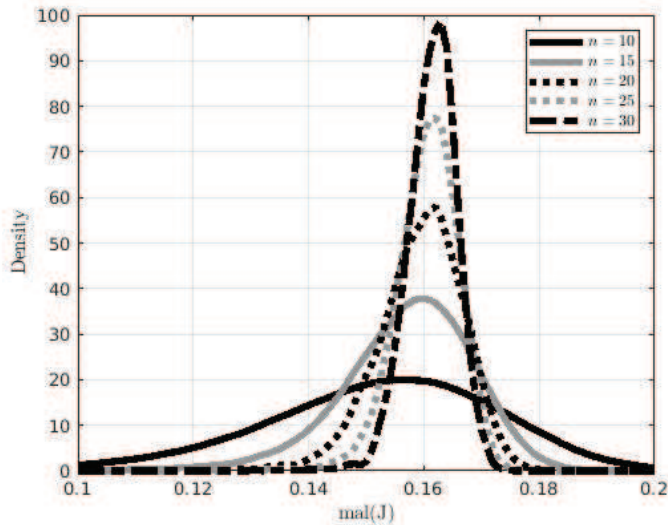}
\caption{$\mal(\mathcal{J}_{n})$ for various n}
\label{fig2}
\end{figure}

From Figure $\ref{fig2}$, it appears that as the dimension increases, the variance of the distribution is shrinking and the center of the distribution is converging to a value around 0.16. Of all the distributions supported in MATLAB's Distribution Fitter application \cite{matlabstatistics}, the Burr Type XII Distribution (a generalized log-logistic distribution) appears to be the best fit for the distributions. Since the distributions appear to converge to a point distribution, it is of interest to see if that is the case, and if so, what the value might be. To address these questions, we performed a regression analysis on the data.

\subsection{Empirical Asymptotics} We used the MATLAB Curve Fitting application \cite{matlabstatistics} to fit a power regression model $p(n) = \alpha n^\beta + \g$ to both the variance and the mean of $\mal(\mathcal{J}_{n})$.

Let $\Var(n)$ denote the variance of the distribution $\mal(\mathcal{J}_n)$ and $\mu(n)$ denote the mean. Then the parameter estimates for the fitted models
%as given in ($\ref{pmodel}$) 
are summarized in Table \ref{summarytable_regression}. 

Thus, empirically at least, the growth of $\Var(n)$ is roughly proportional to $n^{-3}$. Since $0$ is included in the confidence interval for $\g$ in the $\Var(n)$ model, we can reasonably conclude that $\Var(n) \to 0$ as $n \to \infty$. It is important to note that the calculated means and variances for $\mal(\mathcal{J}_{n})$, $n=3,4,5$, are excluded because for $n \geq 6$ the distributions of $\mal(\mathcal{J}_{n})$ appear homogeneous in shape. 
%In a sense, we treat the first few data points as outliers and exclude them from the analysis. 
These analyses are depicted in Figure $\ref{figreg}$. These results suggest the following 
\begin{conj} The limit of $\mu(n)$ exists as $n\to \infty$ and is approximately equal to $0.1654$. Further it holds that for all $\epsilon > 0$ that
    $\bb P\left(|\mal(\mathcal{J}_{n}) - \mu(n) | \geq \epsilon\right) = O\left(\epsilon^{-2} n^{-3}\right).$
\end{conj}

\begin{figure}[t]%
\centering
\includegraphics[scale=0.5]{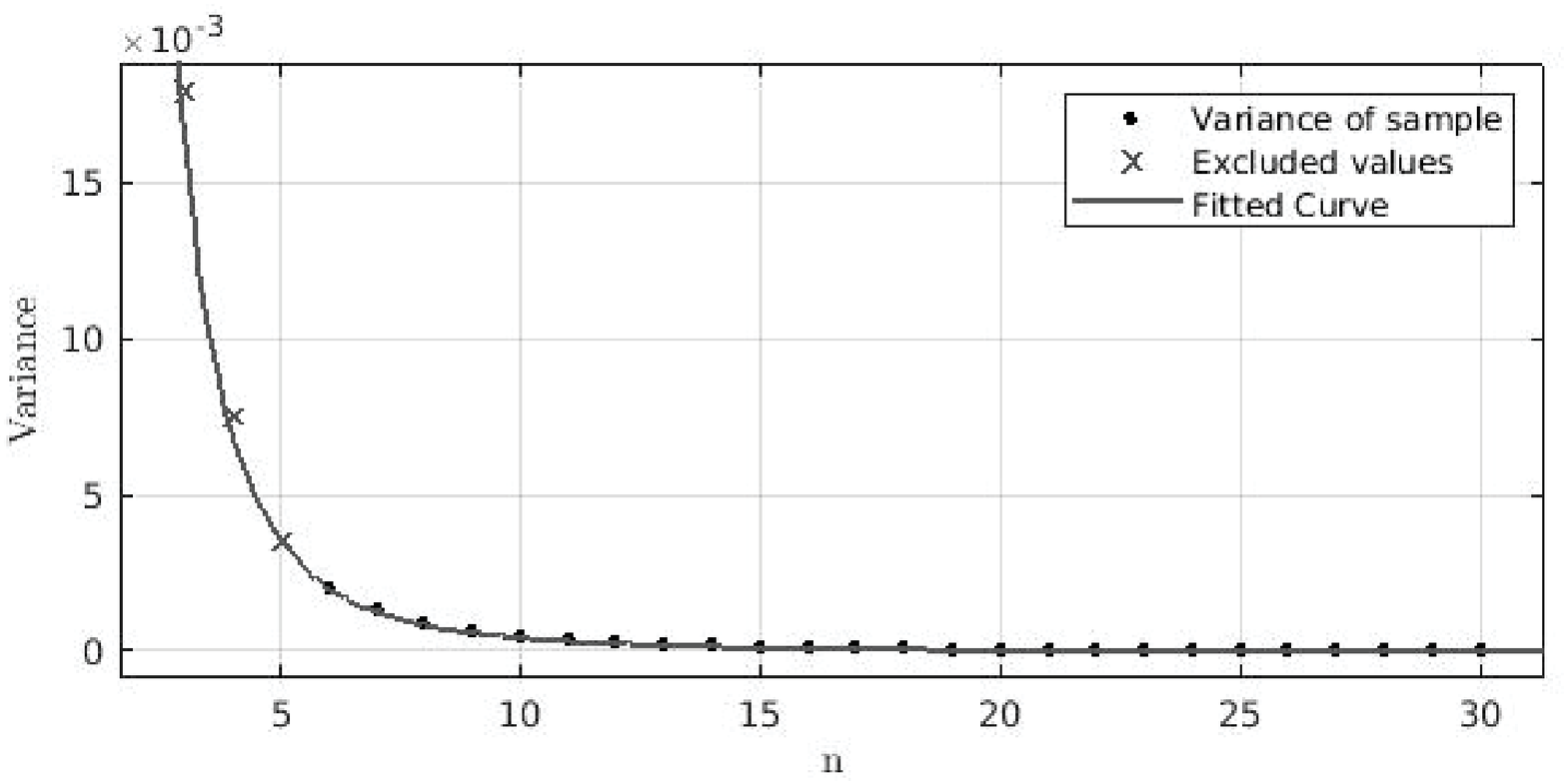}%
\qquad
\includegraphics[scale=0.5]{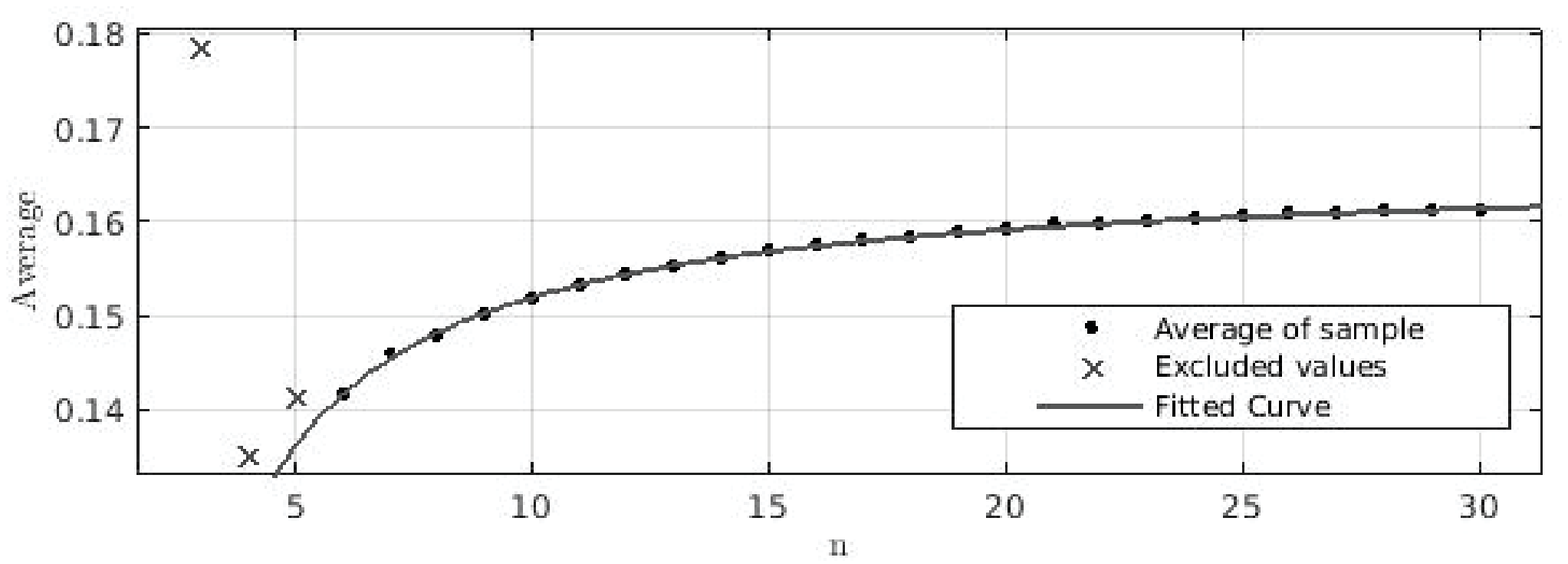}%
\caption{Regression analysis of variance and average of $\mal(\mathcal{J}_n)$}%
\label{figreg}%
\end{figure}

\begin{table}[H]
    \centering
    %\begin{tabular}{|*5{m{25mm}|}}
    \begin{tabular}{|m{18mm}| m{25mm} | m{28mm}| m{25mm} | m{28mm}| }
    \hline
    &\multicolumn{2}{c|}{$\Var(n)$}  &\multicolumn{2}{c|}{$\mu(n)$}\\
    \hline
    Parameter & Point Estimate& 95\% Conf.\ Int.\ & Point Estimate & 95\%  Conf.\ Int.\ \\
    \hline
    $\alpha$ & 0.4176 & (0.4093, 0.4259) & -0.1768 &  (-0.191, -0.1627)\\
    \hline
    $\beta$ & -2.97 & (-2.981, -2.96) & -1.12 & (-1.171, -1.069)\\
    \hline
    $\g$ &  9.861 $\times 10^{-8}$ & ($-1.241 \times 10^{-6}$, $1.439 \times 10^{-6}$) &  0.1654 & (0.165, 0.1659)  \\
    %1.241, 1.439
    \hline
    \end{tabular}
    \caption{Parameter estimates for power regression model $\alpha n^{\beta} + \gamma$}
    \label{summarytable_regression}
\end{table}

\section{Further Problems}

We list several other natural problems and conjectures related to malnormal matrices.

\begin{problem}
Construct a malnormal sequence of matrices $A_n$ with $A_n\in\bb M_n$ for all $n=1,2,3,\dotsc.$
\end{problem}

\begin{problem}
Find a sharp quantitative bound for the constant $\gamma$ in Theorem \ref{main}.
\end{problem}

\begin{problem}
Determine the distribution of eigenvalues for $\cc J_n$. 
\end{problem}

There is a difference in the distribution of eigenvalues of $\mathcal{J}_{n}$ when the set is constructed from orthogonal matrices as opposed to unitary matrices. For convenience, we use $\mathcal{J}_{n}(O)$ and $\mathcal{J}_{n}(U)$ when the set is constructed from orthogonals and unitaries, respectively. The relative frequency of the eigenvalues of $10^5$ matrices from $\mathcal{J}_{10}, \mathcal{J}_{50},$ and $\mathcal{J}_{100}$ are shown in Figures \ref{figorth} and \ref{figunitary}. 
Both $\mathcal{J}_{n}(O)$ and $\mathcal{J}_{n}(U)$ display a clustering of eigenvalues in the corner of a clearly-defined ``pillow'' shaped region, and this clustering intensifies as $n$ increases. Note the band of high probability of real-valued eigenvalues in $\mathcal{J}_{n}(O)$. A similar property is observed in the eigenvalues of the real-valued Ginibre ensemble \cite{Baik2020}. 
The behavior displayed in Figure $\ref{figorth}$ suggests that the limiting distribution of the eigenvalues of $\mathcal{J}_{n}(O)$ is likely complicated, which would make it an interesting problem for further study, especially if the techniques from \cite{Baik2020} could be modified to this situation.

\begin{problem} Let $X_n$ be a random matrix in the (normalized) $n\times n$ Ginibre ensemble, i.e., $X_n = \frac{1}{\sqrt n}(g_{ij})$ where $g_{ij}$ are independent normal complex Gaussian random variables.
Determine the asympotic distribution of $\mal(X_n)$.
\end{problem}

We performed some preliminary data generation to provide a starting point for this problem using matrices with independent real gaussian entries each with variance $1/n$.
We summarize the data generation in Table $\ref{summarytable_malGin}$. Note that $X_n$ is no longer necessarily contractive, but has operator norm close to $2$ with high probability.

\begin{figure}[ht]%
\centering
\includegraphics[width=4cm]{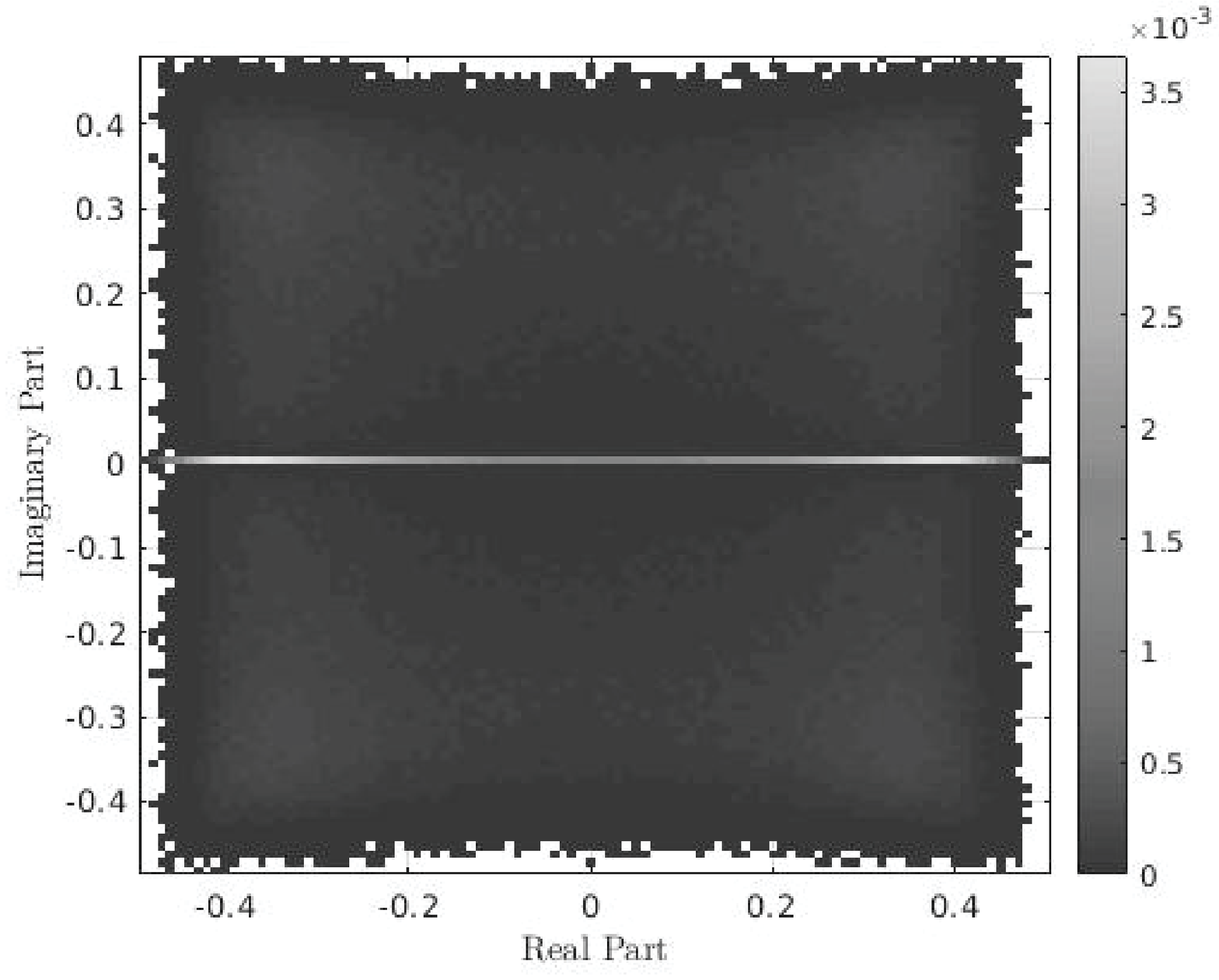}%
\qquad
\includegraphics[width=4cm]{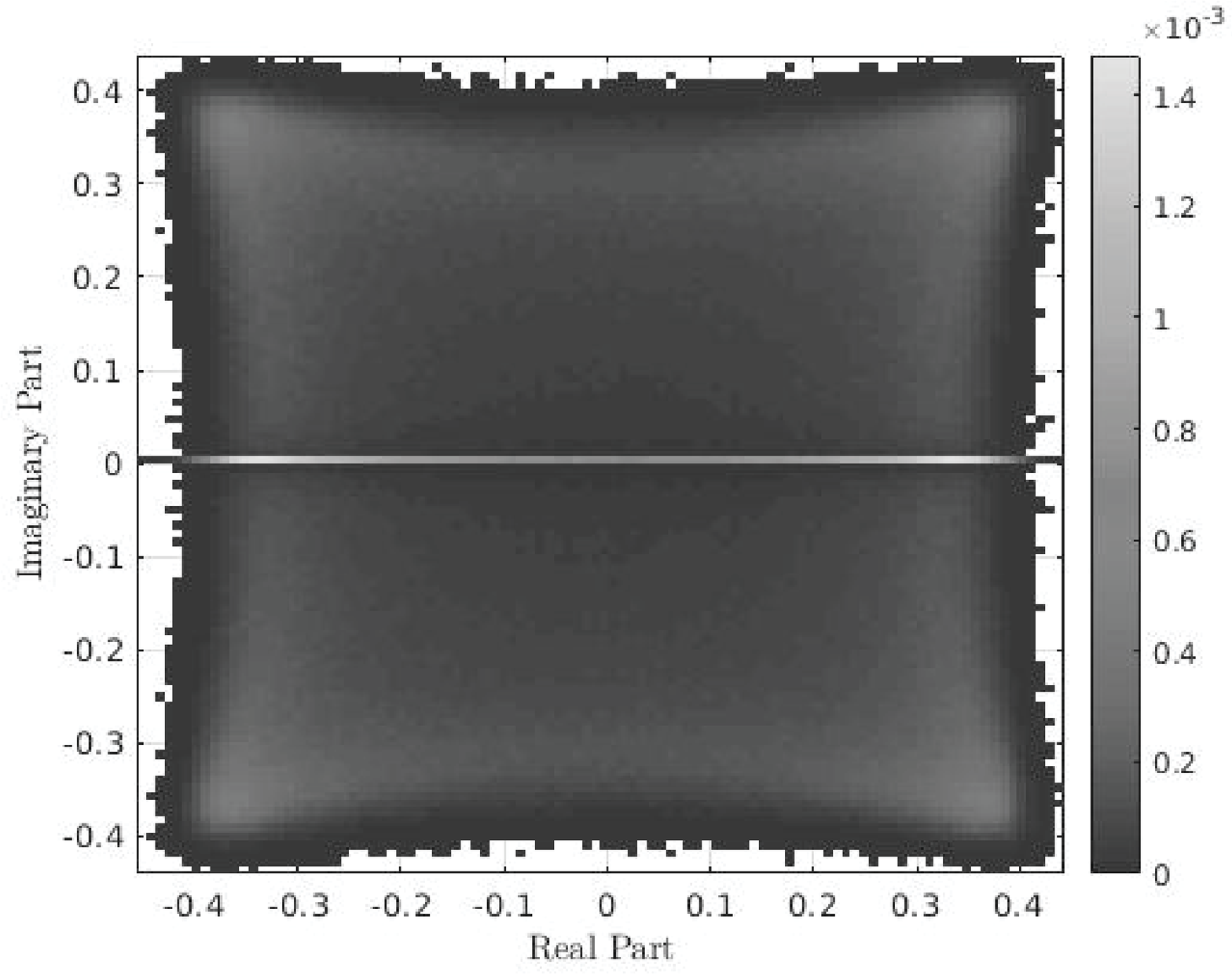}%
\qquad
\includegraphics[width=4cm]{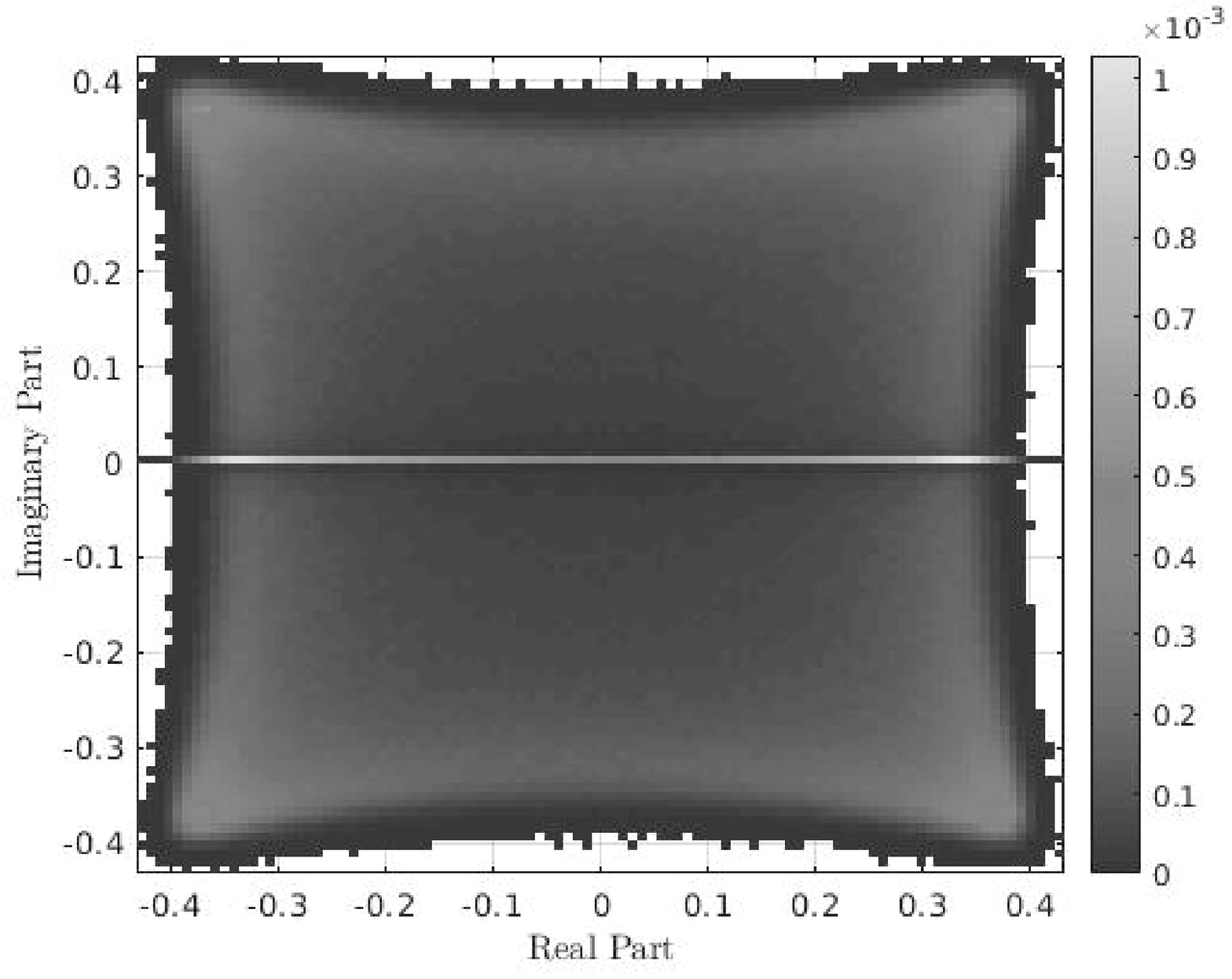}
\caption{Eigenvalues of $10^5$ matrices in $\mathcal{J}_{n}(O)$ for n = 10, 50, 100}%
\label{figorth}%
\end{figure}

\begin{figure}[ht]%
\centering
\includegraphics[width=4cm]{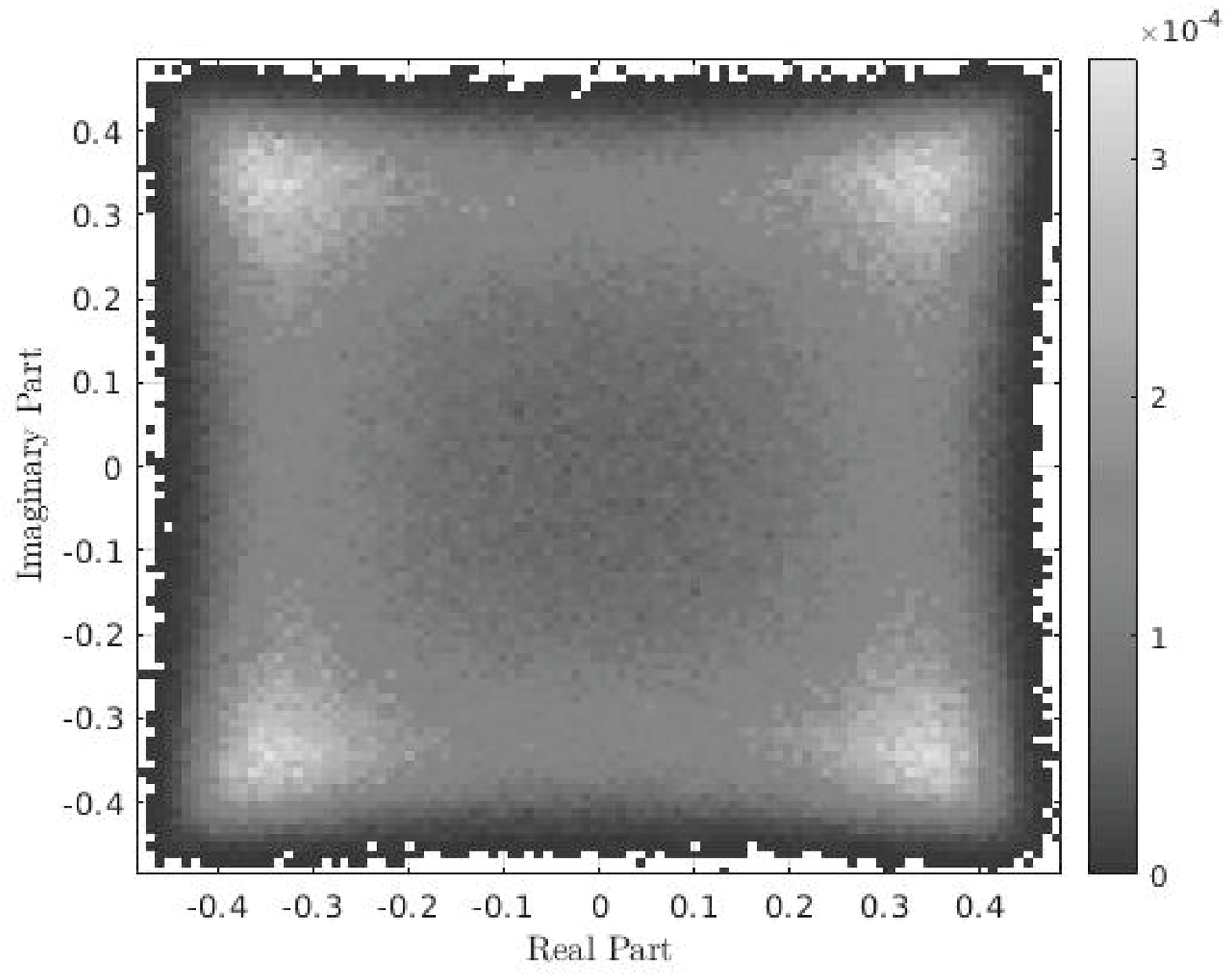}%
\qquad
\includegraphics[width=4cm]{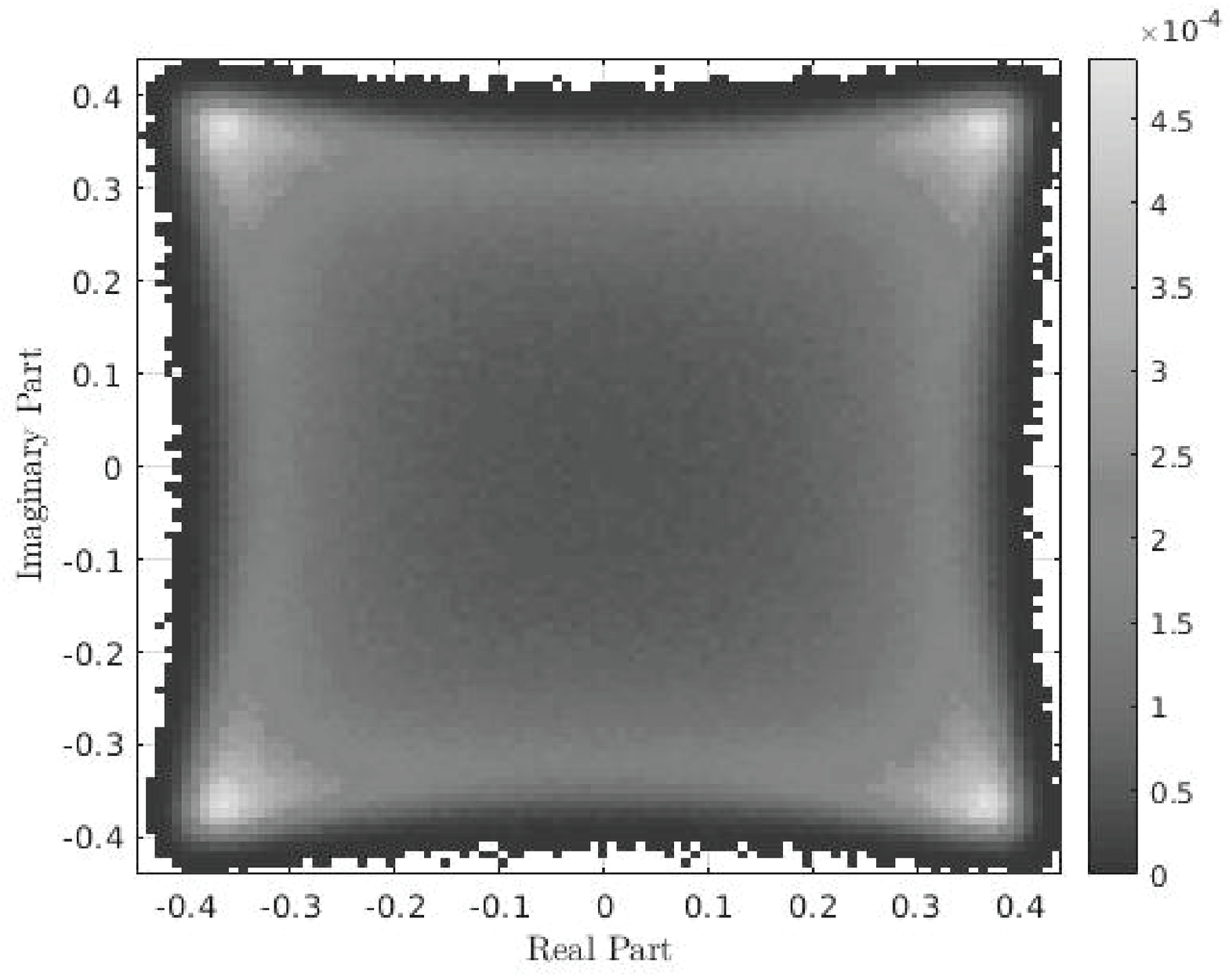}%
\qquad
\includegraphics[width=4cm]{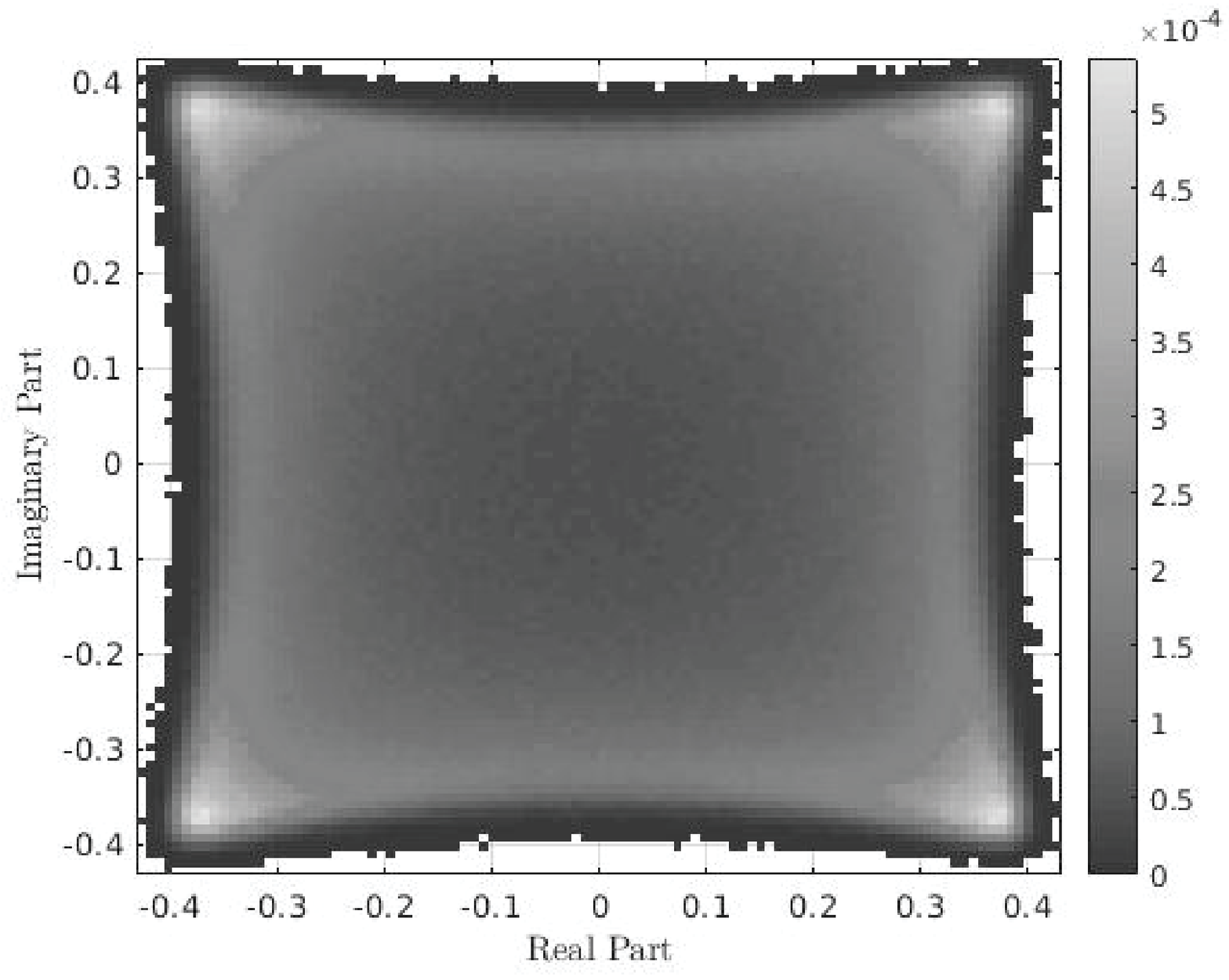}
\caption{Eigenvalues of $10^5$ matrices in $\mathcal{J}_{n}(U)$ for n = 10, 50, 100}%
\label{figunitary}%
\end{figure}

\begin{table} [H]
   \centering
   \begin{tabular}{|m{3em}|m{5em}|m{5em}|m{8em}|m{8em}|} 
   \hline 
   n & Mean & Median & Variance & Size of Sample\\ [0.5ex]
   \hline  5 & 0.35186 & 0.34889 & 0.011115 & $10^6$\\ 10 & 0.33723 & 0.33898 & 0.0020601 & $10^5$\\ 15 & 0.33938 & 0.34114 & 0.00086524 & $10^5$\\ 20 & 0.34218 & 0.3439 & 0.00048062 & $10^4$\\ 25 & 0.34391 & 0.34541 & 0.00031995 & $10^4$\\ 30 & 0.34389 & 0.34559 & 0.00022611 & $10^4$ \\ \hline \end{tabular}
   \caption{Summary statistics of data generated from $\mal(X_{n})$}
   \label{summarytable_malGin}
\end{table}
%%%%%%%%%%%%%%%%%%%%%%%%%%%%%%%%%%%%%%%%%%%%%%%%%%%%%%%%%%%%%%%%%%%
%%%%%%%%%%%%%%%%%%%%%%%%%%%%%%%%%%%%%%%%%%%%%%%%%%%%%%%%%%%%%%%%%%%
%%%%%%%%%%%%%%%%%%%%%%%%%%%%%%%%%%%%%%%

\section*{Acknowledgements} The first author would like to thank Mark Ward for his advice and mentoring, especially in the computational aspects of the project, and Jeya Jeyakumar for  his consultation on the optimization component. The second author would like to thank Ilijas Farah and Bradd Hart for a conversation which sparked his interest in the subject. The second author is also indebted to Ben Hayes and David Sherman for several useful comments and suggestions. G.~Mulcahy was supported by NSF grant DMS-1246818. T.~Sinclair was supported by NSF grants DMS-1600857 and DMS-2055155.

\bibliographystyle{amsalpha}
\bibliography{bibliography}

\end{document}